\documentclass{amsart}
\usepackage{amsmath}
  \usepackage{paralist}
  \usepackage{graphics} 
  \usepackage{epsfig} 
\usepackage{graphicx}  \usepackage{epstopdf}
 \usepackage[colorlinks=true]{hyperref}
\hypersetup{urlcolor=blue, citecolor=red}

\newtheorem{theorem}{Theorem}[section]
\newtheorem{corollary}[theorem]{Corollary}

\newtheorem{lemma}[theorem]{Lemma}
\newtheorem{proposition}[theorem]{Proposition}

\theoremstyle{definition}
\newtheorem{definition}[theorem]{Definition}
\newtheorem{remark}[theorem]{Remark}
\newtheorem{example}[theorem]{Example} 

\newcommand{\deled}{\operatorname{Del}_{\epsilon,\delta}}
\DeclareMathOperator{\PSL}{PSL}
\DeclareMathOperator{\Aut}{Aut}

\usepackage{caption}

\begin{document}

\title{Chaotic delone sets}

\author[J.A. \'{A}lvarez L\'{o}pez]{Jes\'{u}s A. \'{A}lvarez L\'{o}pez}
\address{Jes\'{u}s A. \'{A}lvarez L\'{o}pez, Departamento e Instituto de Matem\'{a}ticas, Facultade de Matem\'{a}ticas, Universidade de Santiago de Compostela, 15782 Santiago de Compostela, Spain}
\email{jesus.alvarez@usc.es}

\author[R. Barral Lij\'o]{Ram\'on Barral Lij\'o}
\address{Ram\'on Barral Lij\'o, Research Organization of Science and Technology,	Ritsumeikan University, Nojihigashi 1-1-1, Kusatsu, Shiga, 525-8577, Japan}
\email{ramonbarrallijo@gmail.com}

\author[J. Hunton]{John Hunton}
\address{John Hunton, Department of Mathematical Sciences, Durham University, Science Laboratories, South Road, Durham, DH1 3LE, UK}
\email{john.hunton@durham.ac.uk}

\author[H. Nozawa]{Hiraku Nozawa}
\address{Hiraku Nozawa, Department of Mathematical Sciences, Colleges of Science and Engineering, Ritsumeikan University,	Kusatsu-Shiga, Japan}
\email{hnozawa@fc.ritsumei.ac.jp}

\author[J.R. Parker]{John R. Parker}
\address{John R. Parker, Department of Mathematical Sciences, Durham University, Science Laboratories, South Road, Durham, DH1 3LE, UK}
\email{j.r.parker@durham.ac.uk}

\begin{abstract}
	We present a definition of chaotic Delone set, and establish the genericity of chaos in the space of $(\epsilon,\delta)$-Delone sets for $\epsilon\geq \delta$. We also present a hyperbolic analogue of the cut-and-project method that naturally produces examples of chaotic Delone sets.
\end{abstract}



\maketitle
\section{Introduction}

This paper is concerned with the relation between chaos theory and the dynamics of Delone sets. Introduced by Delone in the context of mathematical crystallography, Delone sets have been studied also from the viewpoints of arithmetics, topology and foliated spaces. Let us recall the definition of a Delone set and some associated constructions; the reader may consult standard references such as \cite{aperiodic-order-1, Math-long-range-aperiodic-order} for further details about these ideas.

\begin{definition}\label{d. delone}
	Let $\epsilon,\delta >0$. A subset $S$ of a metric space $X$ is $(\epsilon,\delta)$-\emph{Delone} if, 
	\begin{enumerate}[(i)]
		\item for every $x\in X$, there is some $y\in S$ with $d(x,y)\leq \epsilon$  ($S$ is  $\epsilon$-\emph{relatively dense}), and
		\item \label{i. separated}  we have $d(x,y)\geq \delta$ for every  $x,y\in S$, $x\neq y$ ($S$ is $\delta$-\emph{separated}).
	\end{enumerate}
\end{definition}

Given $\epsilon,\delta\in \mathbb{R}^+$, let $\deled$ denote the set of $(\epsilon,\delta)$-Delone subsets of $\mathbb{R}^n$. The set $\deled$ has a canonical, compact, metrisable topology (the \emph{local rubber topology}) such that the action of $\mathbb{R}^n$ given by 
\begin{align*}
	\mathbb{R}^n\times \deled &\longrightarrow \deled\\
	\phantom{\{\, s-v\mid s\in S\,\}:=} (v,S) \!\!\quad\quad  &\longmapsto S-v := \{\, s-v\mid s\in S\,\}
\end{align*}
is a continuous action \cite[Lem.\ 2.5]{BaakeMoody}. Definition~\ref{d. delone}(\ref{i. separated}) makes this action locally free, so that the orbits inherit a canonical smooth structure compatible with the topology.

There is a canonical way of obtaining a dynamical system from such a Delone set \cite[p.~10]{BaakeLenz}. Let $S\in \deled$ and write $[S]$ for the orbit $S+\mathbb{R}^n$. Then $\overline{[S]}$, the closure of $[S]$ in the aforementioned topology, is a compact space endowed with an $\mathbb{R}^n$-action. Roughly speaking, it consists of the Delone sets whose bounded subsets have an approximate replica in $S$; when $S$ is repetitive, these are the Delone sets which are locally indistinguishable from $S$, sometimes called the \emph{local isomorphism class} of $S$ \cite{LagariasPleasants}, but in general $\overline{[S]}$ contains more Delone sets than this local isomorphism class. The main class of Delone sets we consider in this paper will not be repetitive. Since $S$ determines $\overline{[S]}$, we may think of dynamical properties of $\overline{[S]}$ as properties of $S$.

Chaos for group actions is usually characterized by three conditions~\cite{Devaney}: \emph{topological transitivity}, \emph{density of periodic orbits}, and \emph{sensitivity to initial conditions}, of which the first one is trivially satisfied in our situation by the presence of a dense orbit. 
In the case of dynamical systems generated by a continuous map on a metric space, sensitivity to initial conditions follows from the topological transitivity and density of periodic orbits \cite{BanksBrooksCairnsDavisStacey1992}. This result was generalized to continuous actions of topological semigroups on uniform spaces \cite{SchneiderKerkhoffBehrischSiegmund2013}, which directly applies to our setting.  Thus we can omit this condition about sensitivity to initial conditions in our definition of chaos, cf.~\cite{Cairn}. Note that, as detailed in the previous paragraph, we will be dealing with continuous group actions on compact spaces, so the definition of periodic orbit used in~\cite{SchneiderKerkhoffBehrischSiegmund2013} becomes simpler: a Delone set $S$ is \emph{periodic} if the orbit $[S]$ is compact; we may also say that the orbit $[S]$ itself is periodic in this case. This is easily seen to be equivalent to the stabilizer being a lattice in $\mathbb{R}^n$.

This discussion leads us to the following definition, analogous to that in \cite{BarralNozawa}.
\begin{definition}\label{d.chaotic}
	A Delone set $S$ is \emph{almost chaotic} if the union of the periodic orbits is dense in $\overline{[S]}$. We say that $S$ is \emph{chaotic} if it is almost chaotic and \emph{aperiodic}; that is, $S-v\neq S$ for all $v\in\mathbb{R}^n \setminus \{0\}$.
\end{definition}

To the authors' knowledge, such Delone sets have not been studied before. However, the analogous definition in the case of shift spaces  is satisfied for well-known objects, such as subshifts of finite type (see~\cite{Kitchens} for the definition and a nice exposition on the subject). 

Also note that, by simple topological arguments, a  repetitive tiling cannot satisfy the obvious analogous condition. In particular, this immediately rules out examples arising from familiar aperiodic constructions such as primitive substitutions and non-singular canonical Euclidean cut-and-project schemes.

If $S$ is almost chaotic, then $\overline{[S]}$ satisfies the aforementioned requirements of topological transitivity and density of periodic orbits. We require aperiodicity in our definition of chaos because almost chaotic Delone sets include the degenerate case where there is a single compact orbit. 

Recall that a property is \emph{topologically generic} if it holds on a \emph{residual subset}---i.e., a subset  containing a countable intersection of open dense sets. This notion is well-behaved for  \emph{Baire spaces}, which in particular include compact, metrisable spaces  by the Baire Category Theorem.  The first main result of the paper establishes the topological genericity of chaos for $(\epsilon,\delta)$-Delone subsets of $\mathbb{R}^n$ when $\epsilon\geq \delta$.

\begin{theorem}\label{t. main generic}
	If $\epsilon \geq \delta$, then being chaotic is a generic property in $\deled$.
\end{theorem}

This result is similar to that obtained for colored graphs in \cite{BarralNozawa}.
The reason why we impose the condition $\epsilon \geq \delta$ is that it is necessary for extension properties (Lemmas~\ref{l.inner} and~\ref{l. gluing}) that are essential ingredients in our proof. It is also easy to come up with examples where $\epsilon < \delta$ and Theorem~\ref{t. main generic} does not hold---e.g., all $(\delta/2,\delta)$-Delone sets in $\mathbb{R}$ are periodic.

The second aim of this paper is to obtain examples of chaotic Delone sets using a so-called cut-and-project construction on the Poincar\'{e} disk. Being discrete subsets of manifolds, Delone sets lie in a sort of middle ground between geometry and discrete mathematics. There are well-known examples of symbolic dynamical systems satisfying the obvious analogue of Definition~\ref{d.chaotic}---e.g., a two-sided version of Champernowne's number \cite{Champernowne}. A less trivial family of examples comes from the symbolic coding of geodesics in hyperbolic surfaces. This research was initiated by Hadamard in \cite{hadamard} and continued by Morse in \cite{morse1,morse2}, among others. In the particular case of the modular surface, there is an approach for symbolic coding of geodesics that is closer to number theory. In \cite{KatokUgarcovici} the reader can enjoy a nice exposition of these methods and their historical development. All of the aforementioned approaches take advantage of the well-known chaotic properties of the geodesic flow  in compact hyperbolic surfaces to construct chaotic symbolic dynamical systems.

Our method, while related to that described in the previous paragraph, is more geometrical in nature, and naturally yields subsets of $\mathbb{R}^n$ instead of a coding of $\mathbb{Z}$. It is also inspired by the projection method in tiling theory, see~\cite{Forrest}. In our case, we will orthogonally project subsets of an orbit of torsion-free uniform lattices $\Gamma$ in the hyperbolic plane $\mathbb{H}^2$ onto a geodesic. This construction is not guaranteed to produce Delone sets in the general case. We prove a necessary and sufficient condition for this to hold, and present a specific example.

Let us fix a torsion-free uniform lattice $\Gamma$ of $\PSL(2;\mathbb{R})$, a positive number $\rho$ and a point $x$ on $\mathbb{H}^{2}$. For a geodesic $\ell$ on $\mathbb{H}^{2}$, let $p_{\ell} : E_{\ell} \to \ell$ be the orthogonal projection from the open tubular neighbourhood of $\ell$ of radius $\rho$, and define
\[
S_{\ell} = p_{\ell}(E_{\ell} \cap \Gamma x)
\]
(see Figure \ref{fig:delone}). 

\begin{figure}[htbp]
	\begin{minipage}{0.65\hsize}
		\centering
		\includegraphics[width=170px, height=170px] {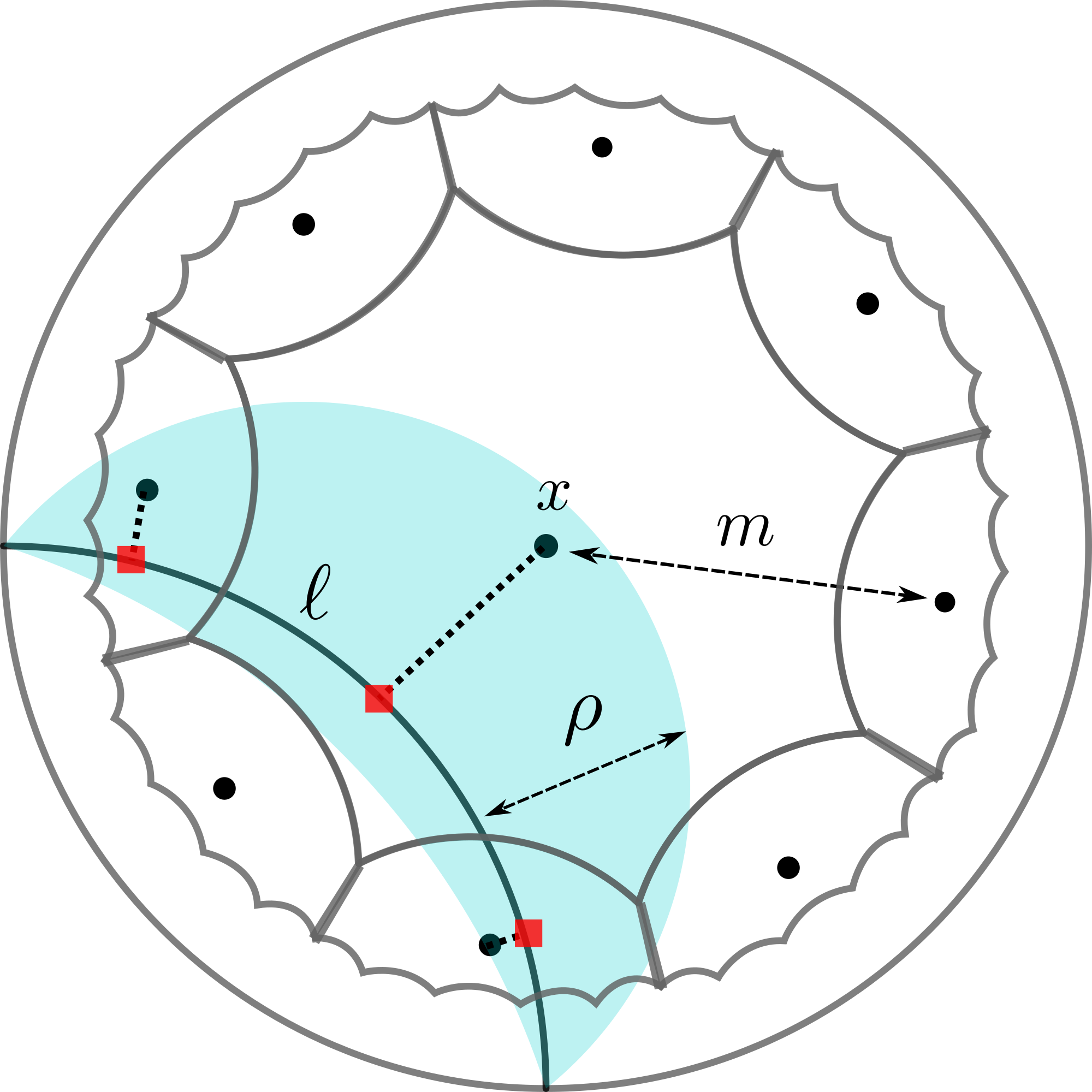}
		\caption{Construction of $S_{\ell}$ in  $\mathbb{H}^2$. The black dots represent points in $\Gamma x$, the blue area is $E_{\ell}$, the red dots represent points in $S_{\ell}$.}
		\label{fig:delone}
	\end{minipage}
	\begin{minipage}{0.30\hsize}
		\centering
		\includegraphics[width=\columnwidth]{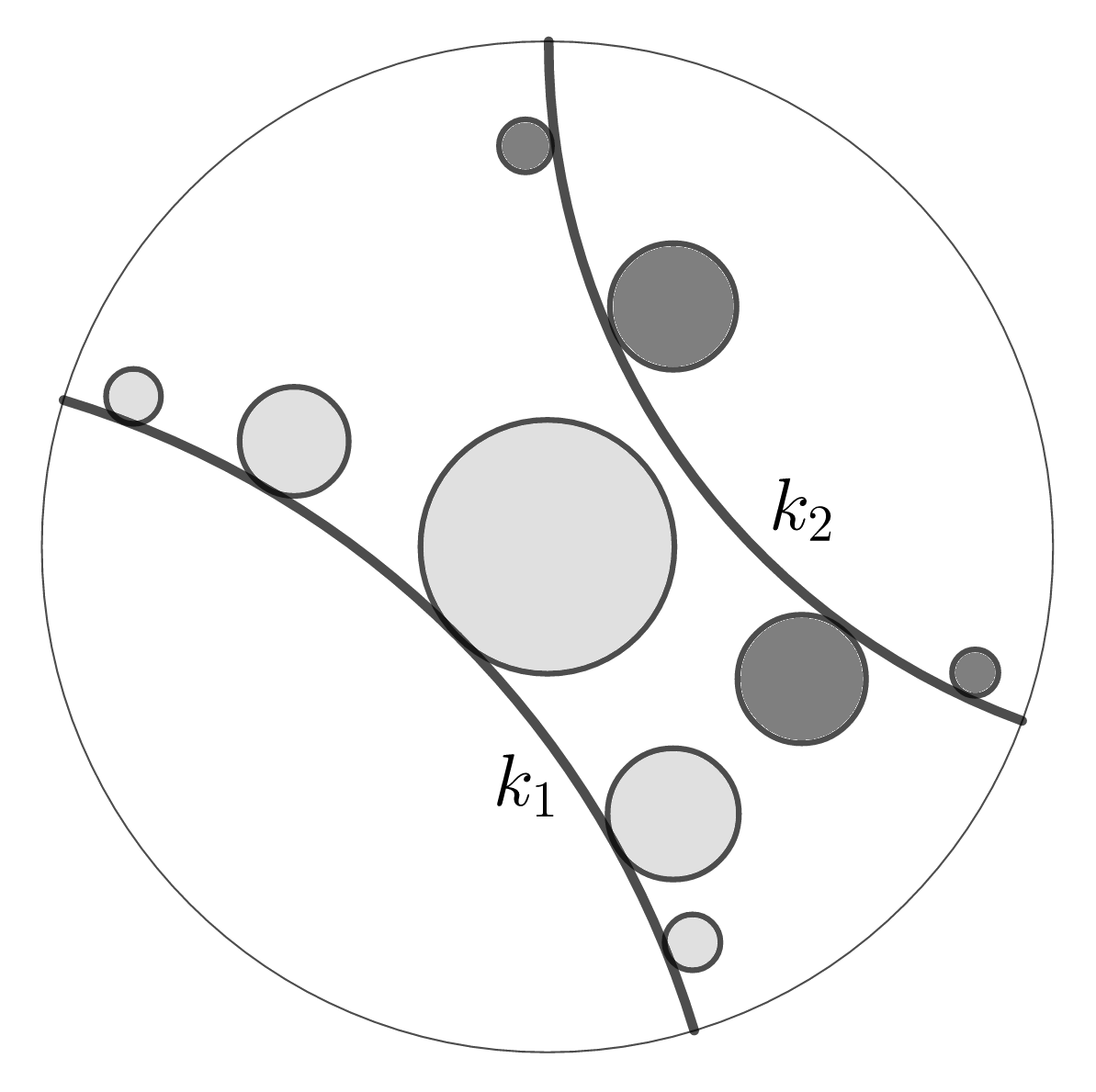}
		\captionsetup{width=0.8\linewidth}
		\caption{The disks represent the inverse image of $\Delta$. The projection of $k_{1}$ to $\Sigma$ has one-sided tangency, while the projection of $k_{2}$ to $\Sigma$ does not.}
		\label{fig:onesided}
	\end{minipage}
\end{figure}

In order to state our result, we need to fix the following terminology: From now on, let $\Sigma = \Gamma\backslash \mathbb{H}^{2}$ be a compact hyperbolic surface. Given a closed disk $D$ on $\Sigma$, a geodesic $\sigma$ on $\Sigma$ is said to have \emph{one-sided tangency with} $\partial D$ 
if $\sigma$ is tangent to $\partial D$ at every point in $\sigma \cap \partial D$, and we can take an orientation of the normal bundle of $\sigma$ so that the outward vector of $\partial D$ at every point of tangency is positive (see Figure~\ref{fig:onesided}). In Section \ref{sec:cp} we prove the following result.

\begin{theorem} 
	\label{thm:cp}
	With the above notation, assume that the orbit of the geodesic flow that consists of the unit tangent vectors of the projection of $\ell$ to $\Sigma$ is dense in $S^{1}(T\Sigma)$, and $d(\ell,y) \neq \rho$ for every $y \in \Gamma x$. Then $S_{\ell}$ is Delone if and only if:
	\begin{enumerate}[(A)]
		\item \label{i:rhoinj} We have $\rho < \operatorname{inj}(\Sigma, x_{0})$. Here $x_{0} = \Gamma x$, and $\operatorname{inj}(\Sigma, x_{0})$ is the injectivity radius of $\Sigma$ at $x_{0}$, which is clearly equal to $\frac{1}{2}\min \{ \, d(y,z) \mid y,z \in \Gamma x,\ y \neq z \, \}$. 
		\item \label{i:geodesicdelta} Any geodesic on $\Sigma$ intersects the closed disk $\Delta$ of radius $\rho$ centred at $x_{0}$, and there exists no geodesic with one-sided tangency with $\partial \Delta$.
	\end{enumerate}
	If $S_{\ell}$ is Delone, then it is chaotic.
\end{theorem}

By Hedlund's theorem (\cite{Hedlund1934}, see also \cite{Hedlund1939} and references therein), the orbits of the geodesic flow that are dense in the unit tangent bundle of $\Sigma$ form a conull set in the space of geodesics. 

It is not easy to check Condition~(\ref{i:geodesicdelta}) in the previous theorem with given $\Gamma$, $\rho$, $x$ and $\ell$, but it is possible for the following example.

\begin{example}
	Let us construct a Riemann surface $\Sigma$ of genus two as follows. Take a hyperbolic $12$-gon $P$ with alternating internal angles $\pi/3$ and $2\pi/3$, all side lengths the same. Identify the sides via the pattern
	\[
	A-B-C-A-D-C-E-D-F-E-B-F
	\]
	going around the boundary (see Figure \ref{fig:P}). There are $3$ orbits of vertices, two made up of three vertices and one made up of $6$. It is easy to see that the quotient has genus $2$ by using the Euler characteristic: $3-6+1=-2$. 
	
	\begin{figure}[htbp]
		\begin{minipage}{0.50\hsize}
			\begin{center}
				\includegraphics[width=\columnwidth]{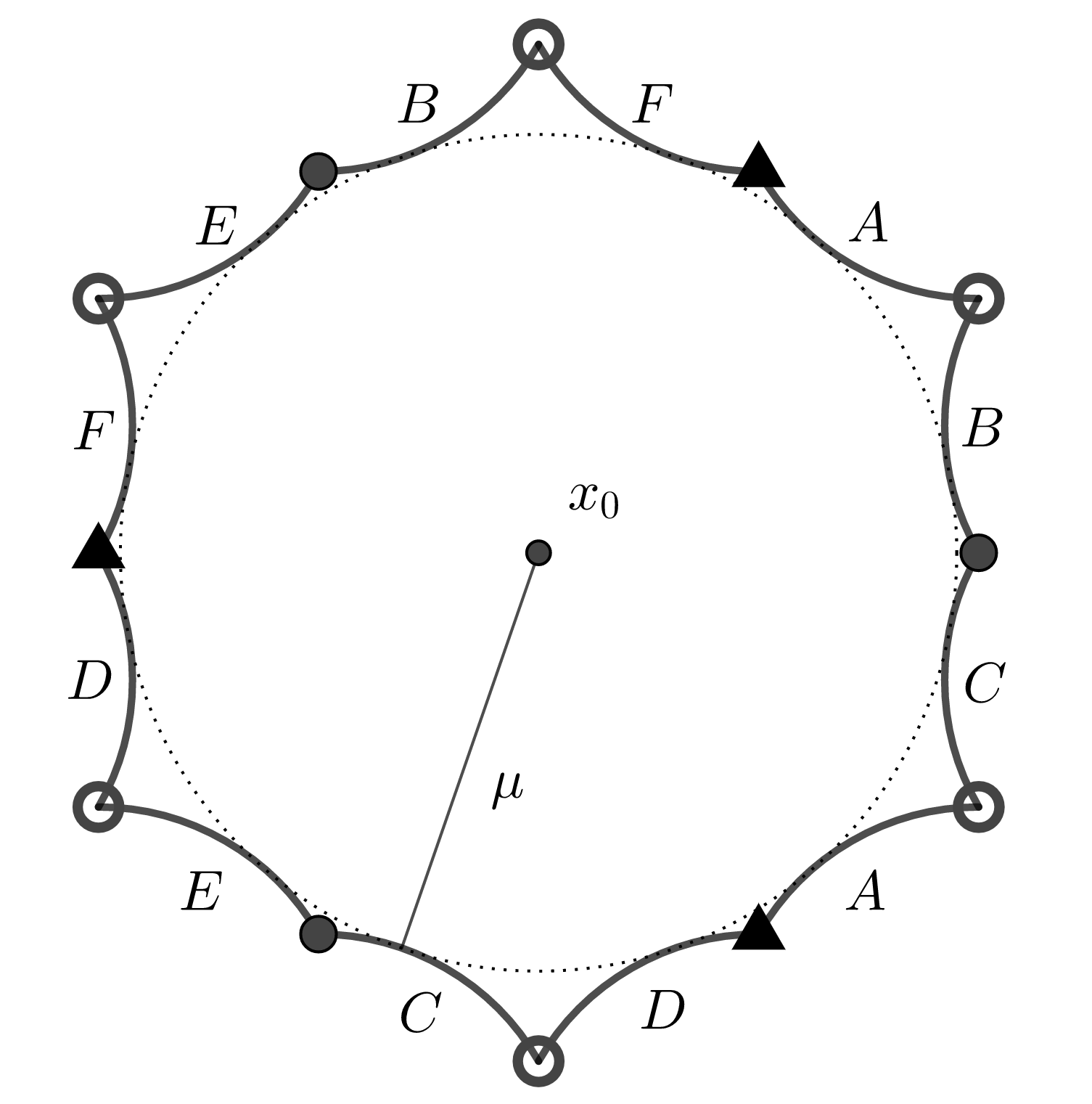}
			\end{center}
			\caption{A $12$-gon $P$}
			\label{fig:P}
		\end{minipage}
		\begin{minipage}{0.45\hsize}
			\begin{center}
				\includegraphics[width=\columnwidth]{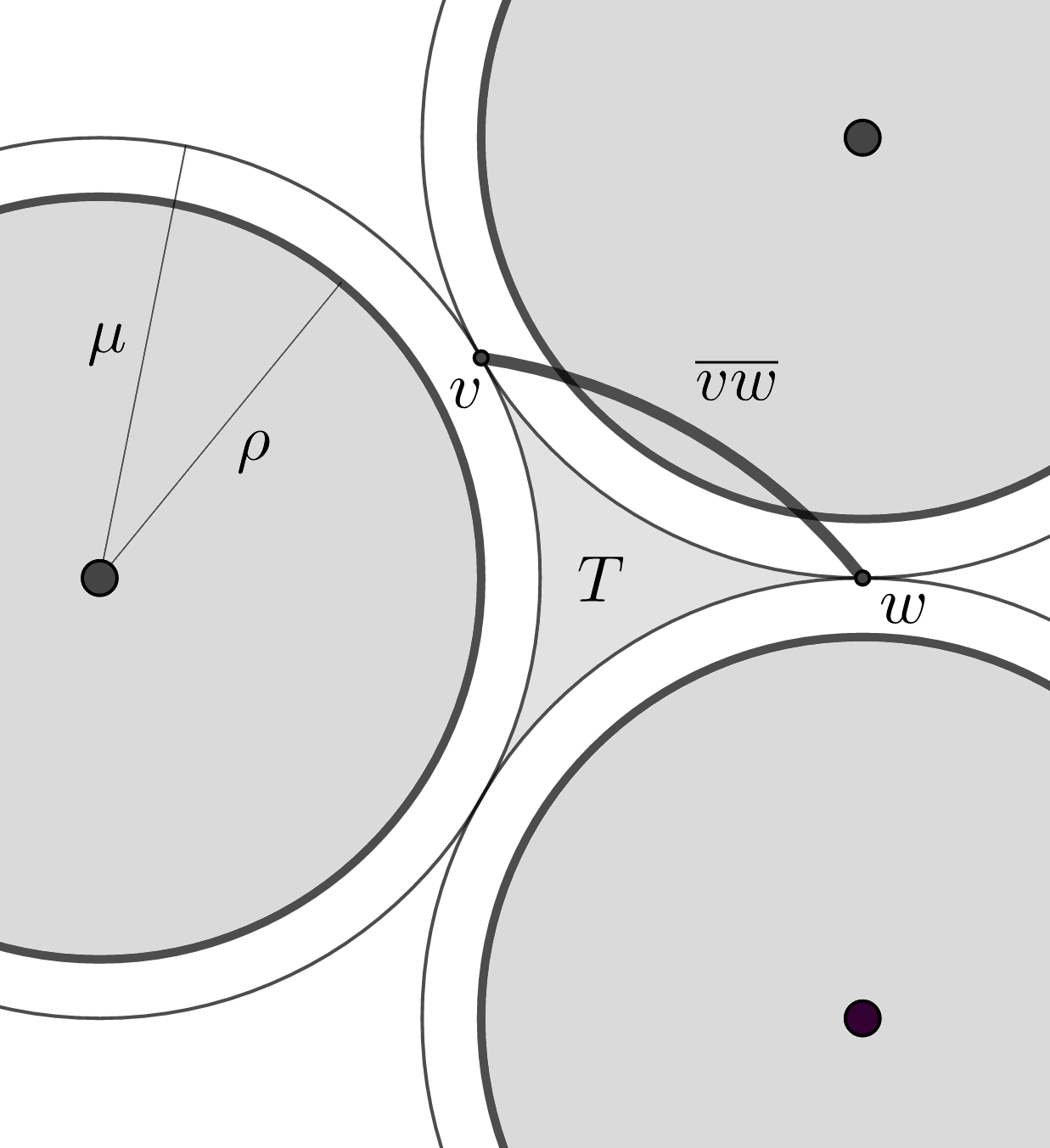}
			\end{center}
			\captionsetup{width=0.8\linewidth}
			\caption{A triangle $T$}
			\label{fig:T}
		\end{minipage}
	\end{figure}
	
	Let $\Gamma < \PSL(2;\mathbb{R})$ be the lattice that corresponds to $\Sigma$. Take $x \in \mathbb{H}^{2}$ so that $x$ is projected to the barycentre $x_{0}$ of $P$. Let $\mu$ denote the injectivity radius of $\Sigma$ at $x_0$. Let $\rho$ be a positive number such that $0 < \mu - \rho \ll 1$. In the sequel we will see that, for any geodesic $\ell$ on $\mathbb{H}^{2}$ that satisfies the assumptions of Theorem~\ref{thm:cp}, the quadruple consisting of $\Gamma$, $x$, $\rho$ and $\ell$ satisfies Conditions~(\ref{i:rhoinj}) and~(\ref{i:geodesicdelta}) in Theorem~\ref{thm:cp}.  Firstly, note that our choice of $\rho$ ensures that  Condition~(\ref{i:rhoinj}) is satisfied. For $r>0$, let $\Delta_{r}$ be the closed disk on $\Sigma$ centred at $x_{0}$ of radius $r$. By the symmetry of the $12$-gon $P$, the disk $\Delta_{\mu}$ is tangent to all edges of $P$. In order to show that Condition~(\ref{i:geodesicdelta}) holds, it is sufficient to show that any geodesic on $\mathbb{H}^{2}$ intersects $\pi^{-1}(\mathring{\Delta}_{\rho})$, where $\pi : \mathbb{H}^{2} \to \Sigma$ is the universal covering projection and $\mathring{\Delta}_{\rho}$ is the interior of $\Delta_{\rho}$. Assume that there exists a geodesic $k$ on $\mathbb{H}^{2}$ contained in $\mathbb{H}^{2} \setminus \pi^{-1}(\mathring{\Delta}_{\rho})$. Here $\pi^{-1}(\partial \Delta_{\mu})$ is a circle packing of $\mathbb{H}^{2}$. Since each angle of $P$ is equal to either of $\pi/6$ or $\pi/3$, we can see that any connected component of $\mathbb{H}^{2} \setminus \pi^{-1}(\Delta_{\mu})$ is either  a triangle or a hexagon. 
	Since each hexagon is adjacent to triangles,
	$k$ intersects a triangle $T$. Since $\rho$ is sufficiently close to $\mu$, the geodesic $k$ should be close to two vertices $v$, $w$ of $T$. Thus $k$ is close to the geodesic segment $\overline{vw}$. 
	Since $\Delta_{\mu}$ is geodesically convex, the segment $\overline{vw}$ is contained in $\pi^{-1}(\Delta_{\mu})$ (see Figure \ref{fig:T}). It follows that $k$ intersects $\pi^{-1}(\mathring{\Delta}_{\rho})$.
\end{example}

It is easy to modify this example to construct an example with $\Sigma$ a closed Riemann surface of arbitrary genus $>1$.

\begin{remark}
	If $\mu \leq \rho$, then $S_{\ell}$ is not $r$-separated for any $r>0$ by the last theorem. But in some cases we can obtain almost chaotic Delone sets in $\mathbb{R}$ or $\mathbb{Z}$ by modifying $S_{\ell}$. We can see that, if $\rho$ is close to $\mu/2$, there cannot be three points in $S_{\ell}$ that are close to each other. Replacing every pair of points which are close to each other with their midpoint, we have a chaotic Delone set in $\ell$.
\end{remark}

Finally, in the last section, we include a short and elementary proof of the fact that, if $S$ is a chaotic Delone set on $\mathbb{R}$, then $S^n$ is a chaotic Delone set on $\mathbb{R}^n$. This shows that we can take products of the above examples to obtain chaotic Delone sets in any dimension.

\section{Preliminaries}\label{sec:prelim}

Let $X$ be a metric space,  let $x\in X$  and $r>0$. We will use $D_X(x,r)$  and $S_X(x,r)$ to denote, respectively,  the \emph{disk} or \emph{closed ball} and the \emph{sphere} of centre $x$ and radius $r$.  We will  omit subscripts when no confusion may arise. 

The canonical topological structure on $\deled$ has received several names, including ``natural topology" \cite{LenzStollman}, ``vague topology" \cite{MullerRichard}, and ``local rubber topology" \cite{BaakeLenz}. Let $\vec{0}\in\mathbb{R}^n$  denote the origin, and let $U$ and $U'$ denote open neighbourhoods of $\vec{0}$,  with $U$  precompact. The local rubber topology mentioned in the introduction  is induced by the entourage base determined by the sets 
\begin{equation}\label{nuu}
	N_{U,U'}:=\{\,(S,S')\in \deled\times \deled \mid S\cap U \subset S'+ U'\text{ and } S'\cap U \subset S+ U' \,\}\;.
\end{equation}
For notational convenience, let
\begin{equation}\label{nr}
	N_r:=N_{B(\vec{0},r),B(\vec{0},1/r)}\qquad  \text{for}\ r>0.
\end{equation} 
For $S\in \deled$, let 
\begin{align*}
	N_{U,U'}(S) & = \{\,S'\in \deled\mid (S,S')\in N_{U,U'}\,\}\;, \\ N_{r}(S) & = \{\,S'\in \deled\mid (S,S')\in N_{r}\,\}\;.  
\end{align*}
For $A,B,C,D$ open neighbourhoods of $\vec{0}$, with $A$ and $B$ relatively compact, one has \cite[p.\ 9]{BaakeLenz}
\begin{equation}\label{eq.composition}
	N_{A+B,B}\circ N_{C+D,D}\subset N_{A\cap C,2(B\cup C)}\;,
\end{equation}
where $2(B\cup C)= (B\cup C)+ (B\cup C)$.

Once we have provided neighbourhood bases for $\deled$, the following lemma follows trivially  from Definition~\ref{d.chaotic}.
\begin{lemma}\label{lem:ac}
	An $(\epsilon,\delta)$-Delone set $S$ is almost chaotic if and only if, for every $r\in \mathbb{N}$, there is a periodic Delone set   $S'\in \deled$   such that  $(S,S')\in N_r$ and, for any $s\in \mathbb{N}$, there is a point $x\in \mathbb{R}^n$ satisfying $(S-x, S')\in N_s$.
\end{lemma}

The following lemmas will be used in the next section. The first one follows by applying Zorn's lemma to $\epsilon$-relatively dense sets (see \'Alvarez-Candel {\cite[Proof of Lemma~2.1]{AlvarezCandel2011}}).
\begin{lemma}\label{l.zorn}
	Every $\delta$-separated subset of $\mathbb{R}^n$ is contained in a $(\delta,\delta)$-Delone set.
\end{lemma}

\begin{lemma}\label{l.inner}
	Let $\epsilon \geq \delta$, let $A\subset \mathbb{R}^n$, and let $S$ be an $(\epsilon,\delta)$-Delone set in $\mathbb{R}^n$. There is an $(\epsilon,\delta)$-Delone set $S'$ on $A$ such that $S$ and $S'$ coincide over the subset 
	\[
	A_\epsilon:=\{\,x\in \mathbb{R}^n\mid D(x,\epsilon)\subset A \,\} \;.
	\]
\end{lemma}
\begin{proof}
	Consider the collection of $\delta$-separated subsets $M$ of $A$ such that $M\cap A_\epsilon=S\cap A_\epsilon$. By Zorn's Lemma, $S\cap A$ is contained in a maximal such subset $S'$. We only need to prove that $S'$ is $\epsilon$-relatively dense in $A$, so let $x\in A$ and let us prove $d(x,S')\leq \epsilon$. If $x\in A_\epsilon$, the assumption that $S$ is a Delone set in $\mathbb{R}^n$ means that there is some $s\in S$ with $d(x,s)\leq \epsilon$. But  $s\in A$ by the triangle inequality and $S\cap A\subset S'$, so $s\in S'$ and $d(x,S')\leq \epsilon$. Consider now the case where $x\in A\setminus A_\epsilon$, and suppose by absurdity that $d(x,S')> \epsilon \geq \delta$. Then $S'\cup \{x\}$ is a $\delta$-separated subset of $M$ strictly containing $S'$ and satisfying $(S'\cup \{x\}) \cap A_\epsilon=S\cap A_\epsilon$. This contradicts the maximality of $S'$, so $d(x,S')\leq \epsilon$.
\end{proof}

\begin{lemma}\label{l. gluing}
	Suppose $\epsilon\geq \delta$, and let $A$ be a subset of either $\mathbb{R}^n$ or $\mathbb{T}^n$. Then, for any $(\epsilon,\delta)$-Delone set $N$ in  $A$, there is an $(\epsilon,\delta)$-Delone set $S$ in $\mathbb{R}^n$ or $\mathbb{T}^n$ such that $S\cap A=N$.
\end{lemma}

\begin{proof}
	We will write the proof for $A\subset\mathbb{R}^n$, the case where $A\subset \mathbb{T}^n$ being identical.
	Consider the collection of subsets $M\subset\mathbb{R}^n \setminus A$ such that $N\cup M$ is $\delta$-separated. By Zorn's Lemma, there is such a subset $L$ that is maximal by inclusion. Then $S:= N \cup L$ trivially satisfies $S\cap A=N$ and is $\delta$-separated by the definition of $N$. Let us prove that it is also  a $\epsilon$-relatively dense, so let $x\in \mathbb{R}^n$. If $x\in A$, then by hypothesis $d(x,N)\leq \epsilon$. If $x\notin A$ and $d(x, S)>\epsilon\geq \delta$, then $S\cup \{x\}$ is $\delta$-separated, contradicting the maximality of $L$.
\end{proof}

\section{Genericity of chaotic Delone sets}

This section contains the proof of Theorem~\ref{t. main generic}. 
We start by proving that aperiodicity is a generic property. 
Let $0<\alpha<\delta/4$ and,  for $q\in \mathbb{Q}^n$, let
\begin{equation}\label{vq}
	V_q= \{\,  S\in \deled \mid \exists x\in S, \, D(x-q,\alpha)\cap S=\emptyset   \, \} \;.
\end{equation}
Intuitively, $V_q$ contains all Delone sets $S$ containing a point $s$ such that $S$ fails to have period $q$ at $s$ with respect to some error parameter $\alpha>0$. We now show that the sets $V_q$ are open and dense and $\bigcap_{q\in\mathbb{Q}^n} V_q$ consists of aperiodic Delone sets.

\begin{proposition}\label{p. Vq open}
	The subsets $V_q\subset \deled$ are open for $q\in \mathbb{Q}^n$.
\end{proposition} 

\begin{proof}
	Let $S\in V_q$, so that there is some $x\in S$ such that $d(x-q,S)=\beta > \alpha$. Let $r\in\mathbb{N}$  be large enough depending on $x$, $q$, $\alpha$, and $\beta$,  and let $S'\in N_r(S)$.  If $r>|x|$, then  the  definition of $N_r(S)$ ensures that there is some $y\in S'$ with $d(x,y)<1/r$. Suppose that there exists some  $z\in B(y-q,\alpha)\cap S'$.  If 
	\[
	r-1/r>|x|+|q|+\alpha\;,
	\]
	then $z\in B(0,r)$. 
	Therefore, by the definition of $N_r(S)$, there is some $z'\in S$ with $d(z,z')<1/r$. We may assume that $\alpha + 2/r<\beta$. Then the triangle inequality yields $d(x-q,z')<\beta$, a contradiction. Therefore $S'\in V_q$ and, since $S'$ was an arbitrary element of $N_r(S)$, we get $N_r(S)\subset V_q$. 
\end{proof}

\begin{proposition}\label{p.vqdense}
	The sets $V_q$ are dense in $\deled$ for $q\in \mathbb{Q}^n$.
\end{proposition}
\begin{proof}
	Let us start by proving that there is some $S\in V_q$ satisfying the condition in~(\ref{vq}) with $x=\vec{0}\in \mathbb{R}^n$. Assume first that $q$ has all coordinates equal to $0$ except the first one. If $|q|+\alpha<\delta $, then any $S\in\deled$ with $\vec{0}\in S$ satisfies the condition in~(\ref{vq}) with $x=\vec{0}$ because it is $\delta$-separated, so assume that $|q|+\alpha\geq \delta $. Let $y=q+(2\alpha,0,\ldots,0)$, and let $S$ be a $(\delta,\delta)$-Delone set containing $\vec{0}$ and $y$, which exists by Lemma~\ref{l.zorn}. Since \[D(q,\alpha)\subset D(y,3\alpha)\subset D(y,\delta)\] by the triangle inequality, we get  that $S$ satisfies~(\ref{vq}) with $x=\vec{0}$. The same strategy applies for general $q\in \mathbb{Q}^n$ after applying a suitable rotation.

	Let us prove that $V_q$ is dense, so let $S'\in \deled$. By Lemma~\ref{l. gluing}, for $r,s\in \mathbb{N}$  and $y$ far enough from $\vec{0}$, there is an $(\epsilon,\delta)$-Delone set $S''$ such that 
	\[
	S'\cap B(\vec{0},r)= S''\cap B(\vec{0},r)
	\] 
	and 
	\[
	y+(S\cap B(\vec{0},s))=S''\cap B(y,s)\;,
	\] 
	where $S$ is the Delone set constructed in the previous paragraph.
	It is clear that, for $s>\delta + \alpha$,  $S''$ satisfies the condition in~(\ref{vq}) with $x=y$. Therefore, given an arbitrary $S'\in \deled$  and $r>0$, we have produced a Delone set $S''\in V_q$ such that $S''\in N_r(S')$, and the proposition follows.
\end{proof}

\begin{proposition}\label{p.vqap}
	The set	$\bigcap\nolimits_{q\in \mathbb{Q}^n} V_q$ consists of aperiodic Delone sets. 
\end{proposition}
\begin{proof} 
	Suppose on the contrary that there are $S\in \bigcap_{q\in \mathbb{Q}^n} V_q$ and $v\in \mathbb{R}^n\setminus\{0\}$ such that $S-v=S$. In particular, this implies that, for every $ q\in \mathbb{Q}^n$ and $z\in S$, $d(z-q,S)\leq |v-q|$. When $|q-v|<\alpha$, we obtain a contradiction with the definition of $V_q$ in~(\ref{vq}).
\end{proof}

\begin{corollary}\label{c.ap}
	Aperiodicity is a generic property in $\deled$ for $\epsilon\geq \delta$.
\end{corollary}
\begin{proof}
	By Propositions~\ref{p. Vq open},~\ref{p.vqdense}, and~\ref{p.vqap}, $\bigcap_q V_q$ is a residual subset consisting of aperiodic Delone sets.
\end{proof}

In order to complete the proof of Theorem~\ref{t. main generic}, we will now show that being almost chaotic is also a generic property.
Let $v_i$, $i=1,\ldots, n$, denote the standard basis of $\mathbb{R}^n$.

\begin{definition}\label{d.w}
	For $m,m'\in \mathbb{N}$, let $W_{m,m'}\subset  \deled$ be the subset of $(\epsilon,\delta)$-Delone sets satisfying the following conditions:
	\begin{enumerate}[(i)]
		\item \label{i.m} there is some $x \in \mathbb{R}^n$ such that	$(S,S-x)\in N_m$, and
		\item \label{i.mprime}for any integer coefficients $a_1,\ldots, a_n$ with $|a_i|\leq m'$ for $i=1,\ldots,n$, we have
		\[
		\Big(S-x,S-x-(m+\delta+\epsilon)\sum\limits_{i=1,\ldots,n} a_iv_i\Big)\in N_{m'}\;.
		\]
	\end{enumerate}
\end{definition}
The intuitive idea behind the definition of $W_{m,m'}$ is as follows: a Delone set $S$ belongs to $W_{m,m'}$ if there is some $x$ such that $S$ is similar to $S-x$ with respect to the parameter $m$, and $S-x$ is close to being a periodic Delone set, where $m'$ measures how close to being periodic $S-x$ is. We will see that $W_{m,m'}$ are open dense sets, and $\bigcap_{m,m'\in\mathbb{N}} W_{m,m'}$ consists of almost periodic Delone sets. 

\begin{proposition}\label{p.wopen}
	The sets $W_{m,m'}$ are open for $m,m'\in \mathbb{N}$.
\end{proposition}

\begin{proof}
	Let $S\in W_{m,m'}$. We will show that there is some $l\in \mathbb{N}$ such that $N_l(S)\subset W_{m,m'}$. By the definition of $W_{m,m'}$, there is some $x\in \mathbb{R}^n$ satisfying Definition~\ref{d.w}(\ref{i.m})--(\ref{i.mprime}). Since the sets $N_r$ are open for $r>0$ and any Delone set in $\mathbb{R}^n$ is locally finite, there are $m>\tilde m>0$ and $\tilde{m}'>m'>0$  such that
	\[
	(S,S-x)\in N_{\tilde m}, \quad \Big(S-x,S-x-(m+\epsilon+\delta)\sum\limits_{i=1,\ldots,n} a_iv_i\Big)\in N_{\tilde {m}'}
	\]
	for $|a_i|\leq m'$, $i=1,\ldots,n$.
	By~(\ref{eq.composition}), we can choose $l$ large enough so that $N_l\circ N_{\tilde m}\circ N_l \subset N_{m}$ and $N_l\circ N_{\tilde{m}'}\circ N_l \subset N_{m'}$. It is now a trivial matter to check that every $S'\in N_l(S)$ satisfies Definition~\ref{d.w}.
\end{proof}

\begin{proposition}\label{p.wdense}
	If $\epsilon \geq \delta$, then the subsets  $W_{m,m'}$ are dense in $\deled$ for $m,m'\in \mathbb{N}$. 
\end{proposition}
\begin{proof}
	Let $S\in \deled$ and $l\in \mathbb{N}$. Identify the $n$-torus $\mathbb{T}^n$ with the quotient of the square $[-m-\delta-\epsilon, m+\delta+\epsilon]^n$ that identifies opposite faces. Let $\pi\colon \mathbb{R}^n \to \mathbb{T}^n$ denote the quotient map. 
	By Lemma~\ref{l.inner} there is a $(\epsilon,\delta)$-Delone set $S'$ on $[-m-\epsilon,m+\epsilon]$ satisfying
	\[
	S'\cap [-m,m]^n = S\cap [-m,m]^n\;.
	\]
	Then $\pi(S'\cap [-m-\epsilon,m+\epsilon]^n)$ is a $\delta$-separated subset and  $\epsilon$-relatively dense in $\pi ([-m-\epsilon,m+\epsilon]^n)$, so applying Lemma~\ref{l. gluing} we may enlarge it to an $(\epsilon,\delta)$-Delone set $T$ on $\mathbb{T}^n$ that satisfies
	\[
	\pi(S\cap [-m,m]^n)= T\cap \pi([-m,m]^n)\;.
	\]
	Choose $x\in \mathbb{R}^n$ sufficiently far from $0$, and lift $T\subset \mathbb{T}^n$ to an $(\epsilon,\delta)$-Delone set $\widehat T$ on a ``grid" of fundamental domains given by the squares with centres $x+\sum a_iv_i$ and length $2(m+\delta+\epsilon)$, as illustrated in Figure~\ref{f.lift}.
	Using Lemma~\ref{l. gluing}, complete the disjoint union 
	\[
	\widehat T\sqcup (S\cap [-l,l]^n) 
	\]
	to an $(\epsilon,\delta)$-Delone set $\widehat S$ satisfying
	\[
	\widehat S \cap [-l,l]^n = S\cap [-l,l] ^n 
	\]
	and
	\begin{multline*}
		\widehat S \cap [x-m'(m+\delta+\epsilon),x+m'(m+\delta+\epsilon)]^n   \\
		\quad \quad=\widehat T \cap [x-m'(m+\delta+\epsilon),x+m'(m+\delta+\epsilon)]^n\;.
	\end{multline*}
	Then $\widehat S$ satisfies the conditions of Definition~\ref{d.w} with $x\in \mathbb{R}^n$. We have shown that, for every $S\in \deled$ and $l\in \mathbb{N}$, there is $\widehat S\in W_{m,m'}\cap N_l(S)$. This establishes the density of $W_{m,m'}$.
\end{proof}

\begin{figure}[tb]
	\includegraphics[scale=0.6]{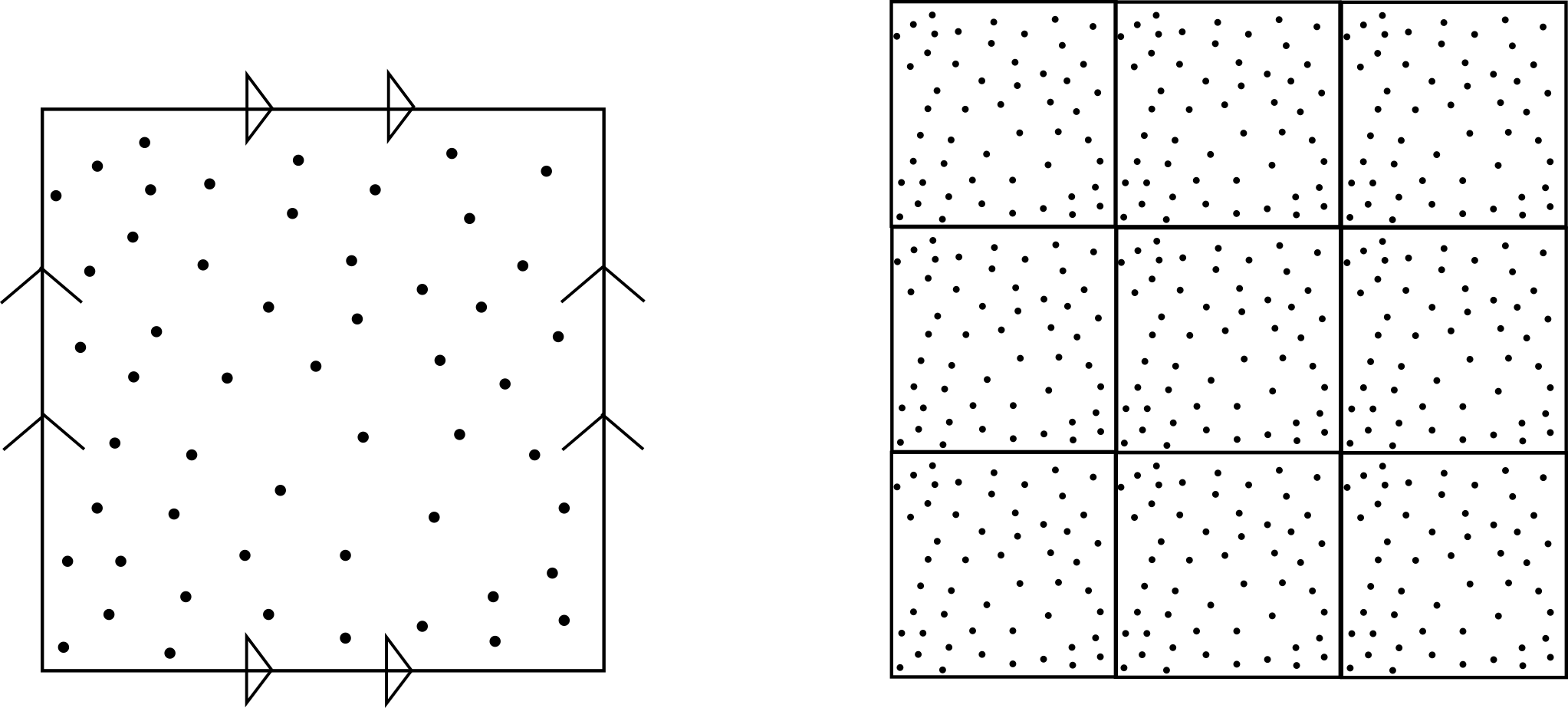}
	\caption{The picture on the left represents $T\subset \mathbb{T}^n$; the right one its lift to $\mathbb{R}^n$ following a grid pattern.}
	\label{f.lift}
\end{figure}

\begin{lemma}\label{l.wac}
	The set $\bigcap_{m,m'} W_{m,m'}$ consists of almost chaotic Delone sets.
\end{lemma}
\begin{proof}
	Let $S\in \bigcap_{m,m'} W_{m,m'}$ and fix a neighbourhood $N_l(S)$ ($l\in \mathbb{N}$).  Let $m>l$. For every $m'$ there is a point $x_{m'}\in \mathbb{R}^n$ such that $(S,S-x_{m})\in N_{m}$ and, for any integer coefficients $a_1,\ldots, a_n$ with $|a_i|\leq m'$, we have 
	\[
	\Big(S-x_{m'},S-x_{m'}-(m+\delta+\epsilon)\sum \limits_{i=1,\ldots,n} a_iv_i\Big)\in N_{m'}\;.
	\]
	Since $\deled$ is compact, the sequence $(S-x_{m'})_{{m'}\in \mathbb{N}}$ has a subsequence converging to some $S'\in \overline{[S]}$, and  $(S,S')\in U_{l}$ because $l<m$. Moreover, for $m'$ large enough and $|a_i|\leq m'$, we have 
	\[
	\Big(S-x_{m'}, S-x_{m'}-(m+\delta+\epsilon)\sum\limits_{i=1,\ldots,n} a_ie_i\Big)\in N_{m'}\;.
	\] 
	By continuity we obtain
	\[
	\Big(S',S'- (m+\delta+\epsilon)\sum\limits_{i=1,\ldots,n} a_ie_i\Big)\in N_{m'}
	\]
	for every $m'\in \mathbb{N}$. This means $(m+\delta+\epsilon)\bigoplus_i a_i\mathbb{Z}^n\subset \Aut(S')$, hence $S'$ is periodic.  We have proved that, for any $S\in \bigcap\nolimits_{m,m'} W_{m,m'}$, there are periodic Delone sets in $\overline{[S]}$ arbitrarily close to $S$, and the result follows.
\end{proof}

\begin{corollary}\label{c.ac}
	Being almost chaotic is a generic property in $\deled$ for $\epsilon\geq\delta$.
\end{corollary}
\begin{proof}
	The set $\bigcap_{m,m'} W_{m,m'}$ is a residual subset consisting of almost chaotic Delone sets by Propositions~\ref{p.wopen} and~\ref{p.wdense} and Lemma~\ref{l.wac}.
\end{proof}

The combination of Corollaries~\ref{c.ap} and~\ref{c.ac} gives Theorem~\ref{t. main generic}.

\section{Cut-and-project construction on the Poincar\'{e} disk}\label{sec:cp}

In this section we will present a geometric example of a chaotic Delone set on $\mathbb{R}$ by proving Theorem~\ref{thm:cp}. 

As we will see in the course of the proof of Theorem~\ref{thm:cp}, it turns out that it is 
more natural to consider a variant of the hyperbolic cut-and-project set $S_{\ell}$ in Theorem~\ref{thm:cp}. Let us fix some notation first: Fix a torsion-free uniform lattice $\Gamma$ of $\PSL(2;\mathbb{R})$, a positive number $\rho$ and a point $x$ in $\mathbb{H}^{2}$ throughout this section. Let $\Sigma = \Gamma\backslash \mathbb{H}^{2}$ be the compact hyperbolic surface obtained from $\Gamma$. From now on, all geodesics on $\mathbb{H}^{2}$ and $\Sigma$ are assumed to be parametrised by arc-length. The image of a geodesic $k : \mathbb{R} \to \mathbb{H}^{2}$ is denoted by the same symbol $k$, and it is identified with $\mathbb{R}$  via the arc-length parametrisation. Thus subsets of the image of geodesics on $\mathbb{H}^{2}$ are regarded as subsets of $\mathbb{R}$. We orient the normal bundle of $k$ with the orientation induced from the standard orientation of $\mathbb{H}^{2}$ and the orientation of $k$. We will consider the following variant  of $S_{\ell}$ in Theorem~\ref{thm:cp}.

\begin{definition}\label{def:Splus}
	Let $k$ be a geodesic on $\mathbb{H}^{2}$. Let $E_{k}$ be the open tubular neighbourhood of $k$ of radius $\rho$ in $\mathbb{H}^{2}$. Let $\partial^{+}{E}_{k}$ be the connected component of the boundary of $E_{k}$ that is positive  with respect to the orientation of the normal bundle of $k$. Let 
	\[
	\overline{E}^{+}_{k} = E_{k} \cup \partial^{+}{E}_{k}\;, \quad S^{+}_{k} = p_{k}(\overline{E}^{+}_{k}\cap \Gamma x),
	\]
	where $p_{k} : \mathbb{H}^{2} \to k$ is the orthogonal projection.
\end{definition}

We fix throughout this section a geodesic $\ell$ on $\mathbb{H}^{2}$ such that the orbit of the geodesic flow consisting of the unit tangent vectors of the projection of $\ell$ is dense in the unit tangent bundle of $\Sigma$. 
As we will see, $S_{\ell}^{+}$ always has a chaotic nature. However, it may not be Delone in general. We will show the following generalization of Theorem~\ref{thm:cp} to $S^{+}_{\ell}$, which characterises when it holds. 

\begin{theorem} 
	\label{thm:cp2}
	With the above notation, $S^{+}_{\ell}$ is Delone if and only if:
	\begin{enumerate}[(A)]
		\item $\rho < \operatorname{inj}(\Sigma, x_{0})$, where $x_{0} = \Gamma x$ and $\operatorname{inj}(\Sigma, x_{0})$ is the injective radius of $\Sigma$ at $x_{0}$.
		\item Any geodesic on $\Sigma$ intersects the closed disk $\Delta$ of radius $\rho$ centred at $x_{0}$, and there exists no geodesic with one-sided tangency with $\partial \Delta$.
	\end{enumerate}
	If $S^{+}_{\ell}$ is Delone, then it is chaotic.
\end{theorem}

This result is slightly more general than Theorem~\ref{thm:cp}. Indeed, in Theorem~\ref{thm:cp} we assume that $d(\ell,y) \neq \rho$ for any $y \in \Gamma x$ which implies that $S_{\ell}^{+} = S_{\ell}$.

First we show the chaotic nature of $S^{+}_{\ell}$. In order to do so, we will use a classical result of Anosov on the 
chaotic nature of the geodesic flow on $\Sigma$.

\begin{theorem}[{\cite{AnosovRussian}, for English translation, see \cite{Anosov}}]\label{thm:Anosov}
	The union of closed orbits is dense in the unit tangent bundle of $\Sigma$.
\end{theorem}

We will say that a geodesic $k$ on $\mathbb{H}^{2}$ is $\Sigma$-\emph{closed} if $k$ is projected on a closed geodesic on $\Sigma$. For a $\Sigma$-closed geodesic $k$, it is easy to see the sets $S_{k}$ and $S^{+}_{k}$ associated with $k$ are periodic. We will prove that $S^{+}_{\ell}$ is almost chaotic by approximating $S^{+}_{\ell}$ with such periodic $S_{k}$ or $S_{k}^{+}$ based on the characterisation of the almost chaotic property in Lemma \ref{lem:ac}. However, if there are $y \in \Gamma x$ such that $d(k,y) = \rho$, it may violate the approximation of $S^{+}_{\ell}$ by $S_{k}$ with $\Sigma$-closed geodesics $k$. As we will see, the set $S^{+}_{k}$ behaves better than $S_{k}$ in this approximation (see Remark \ref{rem:diff}).

In the following lemma we will use $N_{r}$ $(r>0)$ in a situation more general than in Section \ref{sec:prelim}: let $N_{r}$ be the set consisting of all pairs $(T,T')$ of subsets of $\mathbb{R}$ such that 
\[T \cap [-r,r] \subset T' + [-1/r,1/r]\;, \quad T' \cap [-r,r] \subset T + [-1/r,1/r].\]
Now we will show the following, which implies the chaotic nature of $S^{+}_{\ell}$.

\begin{lemma}\label{lem:cha}
	\begin{enumerate}[(i)]
		\item \label{enu:Nr}For any $r > 0$, there exists a $\Sigma$-closed geodesic $k$ such that $(S^{+}_{\ell}, S^{+}_{k})\in N_{r}$.
		\item \label{enu:Ns} For any $s > 0$ and any geodesic $k$ on $\mathbb{H}^{2}$, there exists $a \in \mathbb{R}$ such that $(S^{+}_{\ell}-a, S^{+}_{k})\in N_{s}$.
	\end{enumerate}
\end{lemma}

\begin{proof}
	Take any $r > 0$ and consider the interval $I=\ell ([-r,r])$. Let $v=\frac{d\ell}{dt}\big|_{t=0}$. By Theorem~\ref{thm:Anosov}, we can take a unit vector $w$ tangent to a $\Sigma$-closed geodesic $k$ and  arbitrarily close to $-v$. Let $Z$ be the subset of all points $z$ in $\Gamma x$ such that $d(I, z) \leq \rho$. For $m=k,\ell$, let $\overline{E}^{+}_{m}$ be the union of the open tubular neighbourhood of $m$ of radius $\rho$ in $\mathbb{H}^{2}$ and its positive boundary, as in Definition \ref{def:Splus}. We may assume that the tangent vector $w$ of  $k$  at $t=0$  is sufficiently close to $-v$, so that $I$ is contained in the positive component of $E_{k} \setminus k$ and $J$ is contained in the positive component of $E_{\ell} \setminus \ell$, where $J=k([-r,r])$. Since $Z$ is finite, by replacing $k$ with a $\Sigma$-closed geodesic closer to $I$, we can assume the following: 
	\begin{itemize}
		\item for any $z \in Z$, we have $z \in \overline{E}^{+}_{\ell}$ if and only if $z \in \overline{E}^{+}_{k}$, 
		\item  $d(\iota (y),y) < 1/2r$ for any $y \in J$, where $\iota : J \to I$ is the unique orientation reversing isometry, and
		\item  $d(p_{k}(z),p_{\ell}(z)) < 1/2r$ for any $z \in Z$, where $p_{k} : \mathbb{H}^{2} \to k$ is the orthogonal projection. 
	\end{itemize}
	By the first condition, we have $S^{+}_{\ell} \cap I \subset p_{\ell}(Z)$ and $S^{+}_{k} \cap J \subset p_{k}(Z)$. For any $z \in Z$, by the second and third conditions, we have 
	\[
	d(p_{\ell}(z),\iota(p_{k}(z))) < d( p_{\ell}(z), p_{k}(z)) + d(p_{k}(z), \iota(p_{k}(z))) < 1/r.
	\]
	Since $\iota (\ell(0)) = k(0)$, it follows that $(S^{+}_{\ell}, S^{+}_{k})\in N_{r}$. This completes the proof of (\ref{enu:Nr}).

	For (\ref{enu:Ns}), take any $s > 0$ and any geodesic $k$ on $\mathbb{H}^{2}$. Let $w=\frac{dk}{dt}\big|_{t=0}$. 
	Since the unit tangent vectors of the projection of $\ell$ is dense in $S^{1}(T\Sigma)$ by assumption, we can take $\gamma \in \Gamma$ and a unit tangent vector $v$ of $\ell$ so that $\gamma_{*}v$ is arbitrarily close to $-w$, where $\gamma_{*}$ is the tangent map of the action $\mathbb{H}^{2} \to \mathbb{H}^{2}$ of $\gamma$. Let $I' = k([-r,r])$.  Let $Z'$ be a subset of $\Gamma x$ which consists of all points $z' \in \Gamma x$ such that $d(z',I') \leq \rho$. The rest of the argument is parallel to the proof of (i). Since $Z'$ is finite, by taking $\gamma \in \Gamma$ and the unit tangent vector $v'$ of $\ell$ at parameter $t=a$ so that $\gamma_{*}v'$ is sufficiently close to $-w$, we have $(S^{+}_{\ell}-a, S^{+}_{k})\in N_{s}$.
\end{proof}

\begin{remark}\label{rem:diff}
	The last lemma is not true for $S_{\ell}$ in general. If there exists no $y \in \Gamma x$ with $d(y,\ell) = \rho$, then (\ref{enu:Nr}) is true for $S_{\ell}$. Similarly (\ref{enu:Ns}) is true for a geodesic $k$ such that there exists no $y \in \Gamma x$ with $d(y,\ell) = \rho$.
\end{remark}

\begin{figure}[htbp]
	\includegraphics[width=0.65\columnwidth]{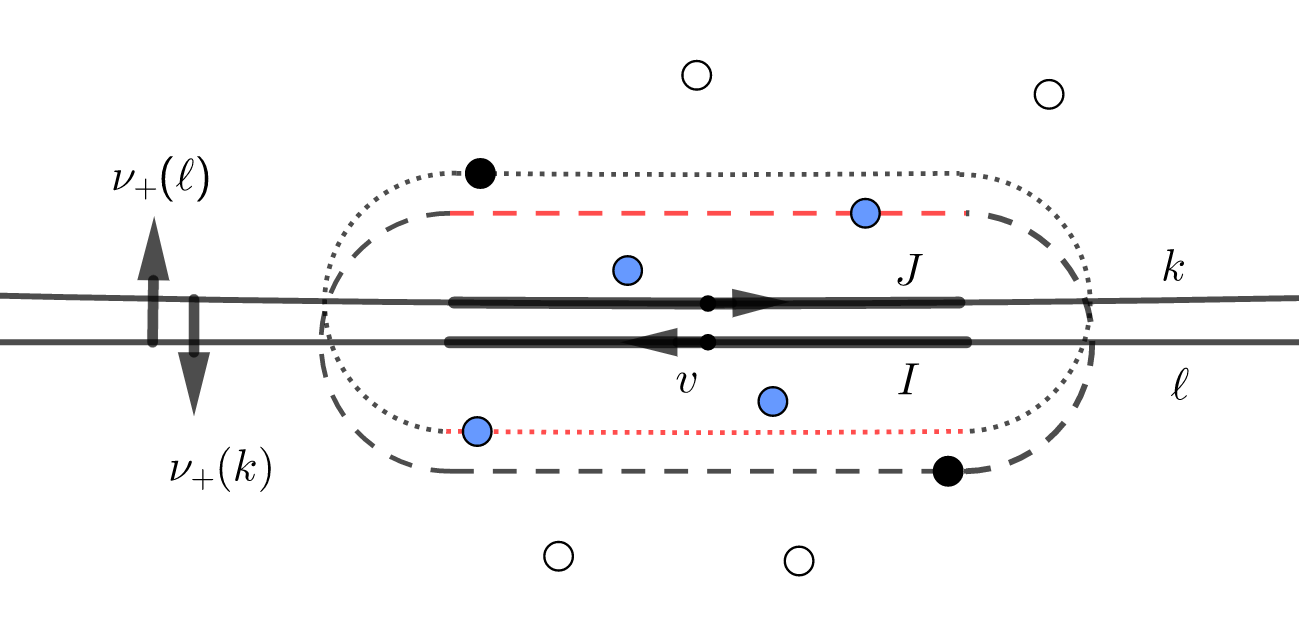}
	\caption{Approximation of $S^{+}_{\ell}$ by $S^{+}_{k}$: The vectors $\nu_{+}(\ell)$ and $\nu_{+}(k)$ represent the orientations of the normal bundles of $\ell$ and $k$, respectively. Two circles with dotted lines represent the boundary of the $\rho$-neighbourhoods of $I$ and $J$, respectively. The dots represent points in $\Gamma x$. The blue dots belong to both $E^{+}_{\ell}$ and $E^{+}_{k}$. But the black dots do not because they belong to the negative side of the boundary of $E_{\ell}$ or $E_{k}$, respectively.}
\end{figure}

Once $S^{+}_{\ell}$ is proved to be Delone, the following consequence of the last lemma shows that $S^{+}_{\ell}$ satisfies the characterisation of an almost chaotic Delone set in Lemma \ref{lem:ac}.

\begin{corollary}\label{cor:ac}
	For every $r\in \mathbb{N}$, there exists a $\Sigma$-closed geodesic $k$ on $\mathbb{H}^{2}$ such that $(S^{+}_{\ell},S^{+}_{k})\in N_r$, and for any $s\in \mathbb{N}$, there exists $a\in \mathbb{R}$ such that $(S^{+}_{\ell}-a, S^{+}_{k})\in N_s$.
\end{corollary} 

Let us characterize now when $S^{+}_{\ell}$ is Delone.

\begin{proposition}\label{prop:charDelone}
	The subset $S^{+}_{\ell}$ is Delone if and only if {\rm Conditions~(\ref{i:rhoinj})} and~{\rm(\ref{i:geodesicdelta})} in {\rm Theorem~\ref{thm:cp2}} are satisfied.
\end{proposition}

Let us prove Proposition \ref{prop:charDelone} by showing the following two lemmas. In the first one, we characterize the discreteness of $S^{+}_{\ell}$ in terms of $\rho$, based on the density of the unit tangent vectors of the projection of $\ell$ in $S^{1}(T\Sigma)$.

\begin{lemma}\label{lem:charsep}
	Let $\mu$ denote the injectivity radius of $\Sigma$ at $x_{0}=\Gamma x$.
	\begin{enumerate}[(i)]
		\item \label{i:sep} If $\rho < \mu$, then $S^{+}_{\ell}$ is $\delta$-separated, where $\delta = 2\mu - 2\rho$.
		\item \label{i:nosep} If $\mu \leq \rho$, then $S^{+}_{\ell}$ is not $\delta$-separated for any $\delta>0$.
	\end{enumerate}
\end{lemma}

\begin{proof}
	First note that $2\mu = \min \{ \, d(y,z) \mid y,z \in \Gamma x,\ y \neq z \, \}$. Here (\ref{i:sep}) follows directly from the triangle inequality. Indeed, for every $y_{i}$ in $S^{+}_{\ell}$, choose $\tilde{y}_{i}\in \Gamma x$ so that $d(\tilde{y}_{i},y_{i}) < \rho$ and $p(\tilde{y}_{i}) = y_{i}$. If $y_{i} \neq y_{j}$, then
	\[
	2\mu \leq d(\tilde{y}_{i},\tilde{y}_{j}) \leq d(\tilde{y}_{i},y_{i}) + d(y_{i},y_{j}) + d(y_{j},\tilde{y}_{j}) < 2\rho + d(y_{i},y_{j}),
	\] 
	which implies that $d(y_{i},y_{j}) > 2\mu - 2\rho = \delta$.

	In order to prove (\ref{i:nosep}), let us assume $\mu \leq \rho$. We consider the case $\mu < \rho$ first. Let $y$ and $z$ be a pair of distinct points in $\Gamma x$ such that $d(y,z) = 2\mu$, and let $v$ be a unit tangent vector at the midpoint of the segment $\overline{yz}$ which is perpendicular to $\overline{yz}$. Let $k$ be the geodesic on $\mathbb{H}^{2}$ such that $\frac{dk}{dt}\big|_{t=0}=v$. Assume that we can take $\gamma \in \Gamma$ so that $\gamma_{*}v$ is very close to a tangent vector of $\ell$ at $t=t_{0}$. Since $\ell (t_{0})$ is close to the midpoint of $\overline{yz}$ and we assume $\mu < \rho$, we have $d(\ell (t_{0}),\gamma(y)) < \rho$ and $d(\ell (t_{0}),\gamma(z)) < \rho$. Hence $p_{\ell}(\gamma(y))$ and $p_{\ell}(\gamma(z))$ belong to $S^{+}_{\ell}$. Since $\ell$ is almost tangent to the bisector of the segment $\overline{\gamma(y)\gamma(z)}$ near the middle point of $\overline{yz}$, we can see that $p_{\ell}(\gamma(y))$ and $p_{\ell}(\gamma(z))$ are close to each other. Since we can take $\gamma \in \Gamma$ so that $\gamma_{*}v$ is arbitrarily close to a tangent vector of $\ell$, it follows that $S$ is not $\epsilon$-separated for any $\epsilon>0$. 
	The case where $\rho = \mu$ follows by a slight modification of the proof. Note that, even if we take a geodesic $k_{1}$ on $\mathbb{H}^{2}$ so that a tangent vector of $k_{1}$ is close to $v$, we may have $d(k_{1},z) > \rho$ or $d(k_{1},y) > \rho$ in general. Instead of approximating $v$ with a tangent vector of $\ell$, first we take a tangent vector $v'$ close to $v$ such that $d(k',y) < \rho$ and $d(k',z) < \rho$, where $k'$ is the geodesic tangent to $v'$. We can take $\gamma \in \Gamma$ so that $\gamma_{*}v'$ is close to a tangent vector of $\ell$. Then, we can do the same argument to see that $p_{\ell}(\gamma(y))$ and $p_{\ell}(\gamma(z))$ are close to each other. 
\end{proof}

Let us characterize the density of $S^{+}_{\ell}$ in the following lemma. In the proof, we say that a geodesic $\sigma$ on $\Sigma$ has \emph{two-sided tangency with} $\partial \Delta$ if  
$\sigma$ is tangent to $\partial D$ at every point in $\sigma \cap \partial D$, but it does not have one-sided tangency with $\partial \Delta$; namely, there exists a pair of outward vectors of $\partial \Delta$ at tangential points in $\sigma \cap \partial D$ that are in the opposite directions.

\begin{lemma}\label{lem:chardense}
	The subset $S^{+}_{\ell}$ is $\epsilon$-relatively dense for some $\epsilon > 0$ if and only if {\rm Condition~(\ref{i:geodesicdelta})} in {\rm Theorem~(\ref{thm:cp2})} is satisfied.
\end{lemma}

\begin{proof}
	The ``only if'' part follows from Lemma \ref{lem:cha}. Indeed, if Condition~(\ref{i:geodesicdelta}) is not satisfied, then there exists a geodesic on $\Sigma$ which does not intersect $\Delta$, or there exists a geodesic on $\Sigma$ with one-sided tangency with $\partial \Delta$. If a geodesic $k$ on $\mathbb{H}^{2}$ does not intersect $\Delta$, then we have $S^{+}_{k} = \emptyset$. If $k$ has one-sided tangency with $\partial \Delta$, then we have $S^{+}_{k} = \emptyset$ after changing the orientation of $k$ if necessary. Since $(S^{+}_{\ell},\emptyset) \in N_{s}$ means that $\ell$ has an interval $I$ of length $2(s-\frac{1}{s})$ such that $I \cap S^{+}_{\ell} = \emptyset$, in any cases, it follows that $S^{+}_{\ell}$ is not $\epsilon$-relatively dense for any $\epsilon > 0$. 
	
	Let us prove the ``if'' part. First consider the case where any geodesic on $\Sigma$ intersects $\mathring{\Delta}$, where $\mathring{\Delta}$ is the open disk of radius $\rho$ in $\Sigma$ centred at $\Gamma x \in \Sigma$. For $v \in S^{1}(T\Sigma)$, let $\tau (v) \in \mathbb{R}_{\geq 0}$  be defined by
	\[
	\tau (v) = \inf \{ \, |t| \in \mathbb{R}_{\geq 0} \mid \ell_{v} (t) \in \mathring{\Delta} \, \}, 
	\]
	where $\ell_{v}$ is the geodesic on $\Sigma$ such that $\frac{d\ell_{v}}{dt}\big|_{t=0}=v$. Since any geodesic intersects $\mathring{\Delta}$, it follows that $\tau : S^{1}(T\Sigma) \to \mathbb{R}_{\geq 0}$ is well-defined. It is easy to see that it is upper semicontinuous. Then, since $S^{1}(T\Sigma)$ is compact, $\tau$ is bounded from above. This implies that $\tau$ is bounded on $\ell$, which implies that $S^{+}_{\ell}$ is $\epsilon$-relatively dense for some $\epsilon$. 
	
	Let us consider the general case. We will show that, if Condition~(\ref{i:geodesicdelta}) in Theorem~\ref{thm:cp2} is satisfied, there are finitely many closed geodesics on $\Sigma$ that have two-sided tangency with $\partial \Delta$, and any other geodesics on $\Sigma$ intersect $\mathring{\Delta}$. Under Condition~(\ref{i:geodesicdelta}) in Theorem~\ref{thm:cp2}, for any geodesic $\sigma$ on $\Sigma$, either $\sigma$ intersects $\mathring{\Delta}$ or $\sigma$ has two-sided tangency with $\partial \Delta$. Since any geodesic sufficiently close to a geodesic with two-sided tangency intersects $\mathring{\Delta}$, the set of unit tangent vectors of $\partial \Delta$ which are tangent to geodesics with two-sided tangency with $\partial \Delta$ is discrete, and hence finite. It follows that there are only finitely many geodesics on $\Sigma$ with two-sided tangency with $\partial \Delta$, and all of them are closed. Let $C$ be the union of closed orbits in $S^{1}(T\Sigma)$ given by the tangent vectors of all geodesics on $\Sigma$ that have two-sided tangency with $\partial \Delta$. Since a geodesic close to a geodesic with two-sided tangency with $\partial \Delta$ intersects $\mathring{\Delta}$, for a sufficiently small open neighbourhood $U$ of $C$, we see that the function $\tau$ is bounded on $U \setminus C$. It follows that $\tau$ is bounded on $S^{1}(T\Sigma) \setminus C$, and hence so is on $\ell$. Then we can conclude that $S^{+}_{\ell}$ is $\epsilon$-relatively dense for some $\epsilon$ as in the above case.
\end{proof}

Proposition \ref{prop:charDelone} follows from Lemmas \ref{lem:charsep} and \ref{lem:chardense}.

Finally, we will show the aperiodicity of $S^{+}_{\ell}$ by applying Lemma \ref{lem:cha} and a result of Dal'bo for the non-arithmeticity of the length spectrum of Riemann surfaces. Recall, the \emph{length spectrum} of a Riemann surface $M$ is the set of the lengths of all closed geodesics on $M$. Dal'bo \cite{Dalbo} proved that the length spectrum of any Riemann surface cannot be of the form $a\mathbb{N}$ for any $a >0$.

\begin{lemma}\label{lem:ap}
	If {\rm Condition~(\ref{i:geodesicdelta})} of {\rm Theorem~\ref{thm:cp2}} is satisfied, then $S^{+}_{\ell}$ is aperiodic.  
\end{lemma}

\begin{proof}
	Assume that $S^{+}_{\ell}$ is periodic with period $\omega$. Take any closed geodesic $\sigma$ on $\Sigma$ and a geodesic $k$ on $\mathbb{H}^{2}$ which is projected to $\sigma$. By assumption, $S^{+}_{k}$ is non-empty. Since $\sigma$ is closed, the set $S^{+}_{k}$ is periodic with period $|\sigma|/m$ for some $m \in \mathbb{N}$, where $|\sigma|$ is the length of $\sigma$. It follows from Lemma \ref{lem:cha}(ii) that $S^{+}_{\ell}$ and $S^{+}_{k}$ have the same period, which means $|\sigma|=\omega m$. Hence, the length spectrum of $\Sigma$ is contained in $\omega \mathbb{N}$. But this contradicts a result of Dal'bo \cite[Proposition 2.1]{Dalbo}.
\end{proof}

Theorem~\ref{thm:cp2} is the combination of Corollary \ref{cor:ac} and Lemma \ref{lem:ap}.

\section{Products of chaotic Delone sets on $\mathbb{R}$}
This section is devoted to the proof of the following result.
\begin{proposition}
	If $S$ is a chaotic Delone subset of $\mathbb{R}$, then $S^n$ is a chaotic Delone subset of $\mathbb{R}^n$ for every $n\geq 1$.
\end{proposition}
\begin{proof}
	Let $S$ be a chaotic $(\epsilon,\delta)$-Delone set for some $\epsilon,\delta>0$, and let $n> 1$. To avoid ambiguity, we denote the elements $\mathbb{R}$ by smallcase letters $x,y,s,\ldots$ and the elements of $\mathbb{R}^{n}$ as vectors $\vec x, \vec y, \vec s,\ldots$. Let 
	\[
	\vec s=(s_1,\ldots,s_n),\ \vec t=(t_1,\ldots,t_n)\in S^n
	\]
	and suppose $\vec s\neq \vec t$, then there is some $1\leq i\leq n$ so that $s_i\neq t_i$. Since $S$ is $\delta$-separated, we have $d_{\mathbb{R}}(s_i,t_i)\geq \delta$, and therefore $d_{\mathbb{R}^n}(\vec s,\vec t)\geq \delta$; this shows that $S^n$ is $\delta$-separated. 
	
	Let us prove that $S^n$ is also $\sqrt{n}\epsilon$-relatively dense: Let $\vec x=(x_1,\ldots,x_n)\in\mathbb{R}^n$. Since $S$ is $\epsilon$-relatively dense, for every $i=1,\ldots,n$, there is some $s_i\in S$ so that $d_{\mathbb{R}}(x_i,s_i)\leq \epsilon$. Let $\vec s = (s_0,\ldots, s_n)$, then 
	\[
	d_{\mathbb{R}^n}(\vec x, \vec s)= \left(\sum_{i=1}^n (x_i-s_i)^2\right)^{1/2}\leq (n\epsilon^2)^{1/2}=\sqrt{n}\epsilon,
	\]
	showing that $S^n$ is a $(\delta,\sqrt{n}\epsilon)$-Delone subset of $\mathbb{R}^n$.
	
	To see that $S^n$ is aperiodic, assume for the sake of contradiction that $S^n-\vec v=S^n$ for some $\vec v=(v_1,\ldots,v_n)\in\mathbb{R}^n$. This means that, for every $\vec s=(s_1,\ldots,s_n)$, $\vec s - \vec v\in S^n$ if and only if $\vec s\in S^n$. In particular, for every $s\in R$, we have $s\in S$ if and only if $s-v_1\in S$, contradicting the hypothesis that $S$ is aperiodic.
	
	Finally, to prove that $S^n$ is almost chaotic, recall that the sets $N_r(S^n)$ ($r>0$) form a neighbourhood basis at $S^n$  (see Section~\ref{sec:prelim}). Also, arguing as before, we get that, for every Delone subset $R$ of $\mathbb{R}$ and $r>0$, 
	\[
	(R+B_{\mathbb{R}}(0,r))^n\subset R^n+B_{\mathbb{R}^n}(\vec 0,\sqrt{n}/r).
	\]
	Now~(\ref{nuu}) and~(\ref{nr}) yield 
	\begin{equation}\label{dimension}
		S\subset N_r(R) \quad \Longrightarrow \quad S^n\subset N_{r/\sqrt{n}}(R^n)
	\end{equation}
	for every $r>0$ and Delone set $R$.
	
	By the assumption that $S$ is almost chaotic and Lemma~\ref{lem:ac}, there is a sequence of periodic Delone sets $T_i$ ($i\geq 1$) in $\mathbb{R}$ and, for each $i$, a sequence $x_{i,j}$ ($j\geq 1$) in $\mathbb{R}$ so that 
	\[
	S\in N_{1/i}(T_i)\quad \text{and} \quad S-x_{i,j}\in N_{1/j}(T_i).
	\] 
	For $i,j\geq 1$, let $\vec x_{i,j}=(x_{i,j},\ldots, x_{i,j})$. Now~(\ref{dimension}) yields  
	\[
	S^n\in N_{\sqrt{n}/i}(T_i^n)\quad\text{and}\quad S-\vec x_{i,j}=(S- x_{i,j})^n \in N_{\sqrt{n}/j}(T_i^n).
	\]
	Arguing as in the beginning of the proof, we get that the sets $T^n$ are Delone, and since they are obviously periodic, the result now follows from Lemma~\ref{lem:ac}.
\end{proof}

\section*{Acknowledgments}
The work was supported supported by the following grants: FEDER/Ministerio de Ciencia, Innovaci\'on y Universidades/AEI/MTM2017-89686-P [to A.L., B.L., and H.N.]; Xunta de Galicia/ED431C 2019/10 [to A.L., B.L., and H.N.]; and JSPS Grant-in-Aid for Scientific Research 17K14195 [to H.N]. It was carried out during the tenure of a Canon Foundation in Europe Research Fellowship by B.L.

\end{document}